\documentclass{amsproc}
\usepackage{eurosym}
\usepackage{amssymb}
\usepackage{amsfonts}
\usepackage[all]{xy}

\theoremstyle{plain}
\newtheorem{theorem}{Theorem}[section]

\newtheorem{claim}[theorem]{Claim}

\newtheorem{corollary}[theorem]{Corollary}

\newtheorem{definition}[theorem]{Definition}
\newtheorem{example}[theorem]{Example}
\newtheorem{lemma}[theorem]{Lemma}
\newtheorem{proposition}[theorem]{Proposition}

\theoremstyle{remark}

\newcommand{\field}[1]{\mathbb{#1}}

\newcommand{\CC}{\field{C}}
\newcommand{\RR}{\field{R}}

\newcommand{\ZZ}{\field{Z}}

\newcommand{\Z}{\field{Z}}

\def\SC{H}

\DeclareMathOperator{\Spec}{Spec}
\DeclareMathOperator{\rank}{rank}
\begin{document}
	
\title[Gorenstein graphic matroids]{Gorenstein graphic matroids}

\author[T. Hibi]{Takayuki Hibi}
\address{Department of Pure and Applied Mathematics,Graduate School of Information Science and Technology,
Osaka University, Suita, Osaka 565-0871, Japan}
\email{hibi@math.sci.osaka-u.ac.jp}

\author[M. Laso\'{n}]{Micha{\l} Laso\'{n}}
\address{Institute of Mathematics of the Polish Academy of Sciences, 00-656 Warszawa, Poland}
\email{michalason@gmail.com}

\author[K. Matsuda]{Kazunori Matsuda}
\address{Kitami Institute of Technology, Kitami, Hokkaido 090-8507, Japan}
\email{kaz-matsuda@mail.kitami-it.ac.jp}

\author[M. Micha{\l}ek]{Mateusz Micha{\l}ek}
\address{Institute of Mathematics of the Polish Academy of Sciences, 00-656 Warszawa, Poland}
\address{Max Planck Institute, Mathematics in Sciences, Inselstrasse 22,
04103 Leipzig, Germany}
\address{Aalto University, Espoo, Finland}
\email{wajcha2@poczta.onet.pl}

\author[M. Vodi\v{c}ka]{Martin Vodi\v{c}ka}
\address{Max Planck Institute, Mathematics in Sciences, Inselstrasse 22,
04103 Leipzig, Germany}
\email{mato.vodicka@gmail.com}

\thanks{Takayuki Hibi was partially supported by JSPS KAKENHI 19H00637. Micha{\l} Laso\'{n} and Mateusz Micha{\l}ek were supported by Polish National Science Centre grant no. 2015/19/D/ST1/01180. Kazunori Matsuda was partially supported by JSPS KAKENHI 17K14165.}

\begin{abstract}
The toric variety of a matroid is projectively normal, and therefore it is Cohen--Macaulay. We provide a complete graph-theoretic classification when the toric variety of a graphic matroid is Gorenstein.
\end{abstract}

\maketitle 

\section{Introduction}

\subsection{Motivation}

Toric varieties is a rich and important class of algebraic varieties defined by the underlaying combinatorial objects. When an algebraic variety is constructed using only combinatorial data, one expects to express its properties by combinatorial means. An attempt to achieve this description often leads to surprisingly difficult combinatorial problems. The case of the toric variety of a matroid is a particularly interesting example.  

The toric variety of a matroid is defined as the toric variety associated to the base polytope of the matroid. White \cite{Wh77} proved that the toric variety of a matroid is projectively normal. Hence, by a celebrated result of Hochster, it satisfies Cohen--Macaulay property. White \cite{Wh80} provided also a conjectural description of generators of the toric ideal of a matroid -- the ideal defining the matroid variety. In particular, he conjectured that the toric ideal of a matroid is generated by quadrics. White's conjecture in full generality remains open since its formulation in $1980$. Blasiak \cite{Bl08} confirmed it for graphical matroids. For general matroids it was proved `up to saturation' \cite{LM14}, see also \cite{La16} for further improvements. The toric variety of a representable matroid (in particular of a graphic matroid) has a nice geometric description -- by a result of Gelfand, Goresky, MacPherson and Serganova \cite{GGMS87} it is isomorphic to the torus orbit closure in a Grassmannian.

In this paper we, in the spirit of algebraic geometry, classify matroids whose toric variety satisfies Gorenstein property. This study was initiated by Herzog and Hibi \cite{HeHi02} who observed that Gorenstein property partitions matroids in an interesting way. As they doubt at full classification (among all matroids), we focus on graphic matroids -- one of the basic classes motivating the concept of matroid. We provide a complete classification of graphic matroids whose toric variety is Gorenstein. We believe that our study, apart from algebraic meaning, is also interesting from its own combinatorial sake.

\subsection{Gorenstein rings  and polytopes}

The Gorenstein property of a ring was introduced by Grothendieck. It reflects many symmetries of cohomological properties of the ring. It is also a condition which implies that the singularities of the spectrum of the ring are not `too bad'.  In particular, regular rings and complete intersections are Gorenstein, and furthermore Gorenstein rings are always Cohen--Macaulay.
We do not present a formal definition of the Gorenstein property, as it is out of scope of the article. Instead, we provide several equivalent definitions when a semigroup algebra over a normal polytope $P$ is Gorenstein. 

Let $P\subset \RR^d$ be a normal lattice polytope (i.e. vertices of $P$ are integer, and every point from $kP\cap\ZZ^d$ is a sum of $k$ points from $P\cap\ZZ^d$) and let $\CC[P]$ be the semigroup algebra for the semigroup of lattice points in the cone $C(P)$ over $P$, which is a standard toric construction \cite{Ful93,CLS11,Stks,BG}. Then the following conditions are equivalent \cite{Ba94,Hi92}:
\begin{enumerate}
	\item $\CC[P]$ is a Gorenstein algebra,
	\item the canonical divisor of $\Spec \CC[P]$ is Cartier,
	\item the numerator of the Hilbert series of $\CC[P]$ is palindromic,
	\item in a multiple of $P$ there is a lattice element $v\in\delta P$, such that the polytope $\delta P-v$ is a reflexive polytope.
\end{enumerate}

Recall that a lattice polytope is \emph{reflexive} if its dual polytope is also a lattice polytope. By virtue of the work of Batyrev \cite{Ba94} reflexive polytopes play an important role in mirror symmetry and are very intensively studied,  e.g. \cite[Section 8.3]{CLS11} and references therein. Furthermore, relations of the Gorenstein property to combinatorics are also an important research topic \cite{BR07, HO06, HHO11, HJM18, HKM17, St91}. 

We will use the following description of the Gorenstein property:

\begin{definition}\label{def:charGor}
For a positive integer $\delta$, a full-dimensional normal lattice polytope $P\subset \RR^d$ is \emph{$\delta$-Gorenstein} if there exists a lattice point $v\in\delta P$ such that for every supporting hyperplane of the cone over $P$ its reduced equation $h$ (i.e. $h$ such that $h(\ZZ^d)=\ZZ$) satisfies $h(v)=1$. A normal lattice polytope $P$ is \emph{Gorenstein} if $P$ is $\delta$-Gorenstein (in the lattice it spans affinely) for some positive integer $\delta$.
\end{definition}

It is straightforward to argue:

\begin{lemma}\label{product}
	Let $P_1,P_2$ be two normal lattice polytopes. The product polytope $P_1\times P_2$ is $\delta$-Gorenstein if and only if both $P_1$ and $P_2$ are $\delta$-Gorenstein.
\end{lemma}

\begin{proposition}\label{prop:charGor}
	Let $P$ be a normal lattice polytope. The semigroup algebra $\CC[P]$ is a Gorenstein algebra if and only if $P$ is Gorenstein.
\end{proposition}

\begin{proof}
When $P$ is $\delta$-Gorenstein, then vertices of the dual polytope $(\delta P-v)^*$ correspond to reduced equations of supporting hyperplanes of the cone over $P$, and therefore they are integer. Conversely, if $h(v)>1$ for the reduced equation of a supporting hyperplane of the cone over $P$, then the coordinates of the corresponding vertex of the dual polytope $(\delta P-v)^*$ are fractions with denominator $h(v)$, hence they are not integer. 
\end{proof}

\subsection{Matroid polytopes and graphic matroids}

Let $M$ be a matroid on a ground set $E$ with the set of bases $\mathfrak{B}$ and the set of independent sets $\mathfrak{I}$ (we refer the reader to \cite{Ox11} for a comprehensive monograph on matroids). 

\begin{definition}
The \emph{base polytope} of a matroid $M$, denoted by $B(M)$, is the convex hull of all bases indicator vectors in $\ZZ^{|E|}$, that is $e_B:=\sum_{b\in B} e_b$ for $B\in \mathfrak{B}$.\newline
The \emph{independence polytope} of a matroid $M$, denoted by $P(M)$, is the convex hull of all independent sets indicator vectors in $\ZZ^{|E|}$, that is $e_I:=\sum_{i\in I} e_i$ for $I\in \mathfrak{I}$.
\end{definition}

Notice that $B(M)$ is contained in a hypersimplex, and $B(M)$ is a facet of $P(M)$. Moreover, every edge of a matroid base polytope $B(M)$ is a parallel translate of $e_{i}-e_{j}$ for some $i,j\in E$ \cite{GGMS87}. In other words, the edges of $B(M)$ correspond exactly to the pairs of bases that satisfy the symmetric exchange property (see \cite{La15} for more exchange properties). As already mentioned both polytopes $B(M)$ and $P(M)$ are normal \cite{Wh77}. 

\begin{definition}
Let $G=(V,E)$ be a finite undirected (multi-)graph. The \emph{graphic matroid} corresponding to $G$, denoted by $M(G)$, is a matroid on the set $E$ whose independent sets are the forests in $G$. A set $B\subset E$ is a basis of $M(G)$ if and only if $B$ is a spanning forest of $G$. 
\end{definition}

\begin{example}
	Consider $3$-cycle $C_3$. The base polytope of the graphic matroid of $C_3$, that is $B(M(C_3))$, is a simplex spanned by three vertices $(1,1,0),(1,0,1),(0,1,1)$. Thus the algebra $\CC[B(M(C_3)]$ is a polynomial ring, in particular it is Gorenstein. 
\end{example}

Recall that a matroid is called \emph{connected} if it is not a direct sum of two matroids. In particular, graphic matroid $M(G)$ is connected if and only if $G$ is $2$-connected. Moreover, the base polytope (and the independence polytope) of a matroid is the product of base polytopes (correspondingly independence polytopes) of its connected components. Hence, the base polytope (and the independence polytope) of $M(G)$ is the product of base polytopes (correspondingly independence polytopes) of graphic matroids corresponding to $2$-connected components of $G$. 

By the above remark and Lemma \ref{product}, the polytope $B(M(G))$ (and $P(M(G))$) is $\delta$-Gorenstein if and only if for every $2$-connected component $H$ of $G$ the polytope $B(M(H))$ ($P(M(H))$) is $\delta$-Gorenstein. Therefore in the proceeding sections we restrict to $2$-connected graphs. 

If $e$ is a loop in a matroid $M$, then polytopes $P(M)$ and $P(M\setminus\{e\})$ (also $B(M)$ and $B(M\setminus\{e\})$) are lattice isomorphic. Thus, $P(M)$ is Gorenstein if and only if $P(M/\{e\})$ is. Since adding or subtracting loops does not affect the property of being Gorenstein, we restrict to loopless matroids.

\subsection{Our results}

We provide a complete graph-theoretic classification of graphs $G$ for which the semigroup algebra of lattice points in the base polytope of the graphic matroid of $G$, that is $\CC[B(M(G))]$, is a Gorenstein algebra.
Theorems \ref{TranslationB}, \ref{thm:d>2} and \ref{2-characterization} add up to the following result.

\begin{theorem}\label{thm:d>2}
	Let $G$ be a finite simple undirected graph. The following conditions are equivalent:
	\begin{enumerate}
		\item $\CC[B(M(G))]$ is a Gorenstein algebra,
		\item there exists a positive integer $\delta$ such that every $2$-connected component of $G$ can be obtained, when	$\delta>2$ using constructions from Propositions \ref{construction1} and \ref{construction2} from a $\delta$-cycle, when $\delta=2$ can be obtained using construction from Proposition \ref{construction3} from the clique $K_4$.
	\end{enumerate} 
\end{theorem}

The same type of result we obtain for the algebra associated to the independence polytope $\CC[P(M(G))]$ corresponding to a multigraph $G$. Theorems \ref{TranslationP}, \ref{caseP} add up to the following.

\begin{theorem}
	Let $G$ be a finite undirected multigraph. The following conditions are equivalent:
	\begin{enumerate}
		\item $\CC[P(M(G))]$ is a Gorenstein algebra,
		\item there exists an integer $\delta\geq 2$ such that every $2$-connected component of $G$ is a $(\delta-1)$-blow up of a graph that is $\delta$-chordal (any cycle without a chord has exactly $\delta+1$ elements) and has no $K_4$ minor,
		\item there exists an integer $\delta\geq 2$ such that every $2$-connected component of $G$  is a $(\delta-1)$-blow up of a graph that can be constructed from the clique $K_2$, by at each step adding a new $(\delta+1)$-cycle to an edge of the preceding graph.
	\end{enumerate} 
\end{theorem}

The case of the independence polytope is way much easier compared to the base polytope. We study independence polytopes in section \ref{2}, the remaining sections \ref{3}, \ref{4}, and \ref{5} treat base polytopes.
		
\section{Characterization of Gorenstein polytopes $P(M(G))$}\label{2}

In this section we characterize multigraphs $G$ for which $P(M(G))$ is Gorenstein. As we will see such multigraphs are related to graphs, by the blow up construction.

\begin{definition}
The $m$-blow up of a graph is a multigraph on the same set of vertices, in which every edge of the graph is replaced by $m$ parallel edges.
\end{definition}

The following theorem is a corollary of a result of Herzog and Hibi \cite{HeHi02} for the case of graphic matroids.

\begin{theorem}\label{TranslationP}
	Let $G$ be a multigraph. The polytope $P(M(G))$ is Gorenstein if and only if there exists an integer $\delta\geq 2$ such that $G$ is a $(\delta-1)$-blow up of a graph $H=(V,E)$ satisfying condition $(\clubsuit)_{\delta}$: $$(\delta-1)\lvert E(S)\rvert+1=\delta(\lvert S\rvert-1),$$
	for every set $S\subset V$ inducing a $2$-connected subgraph.
\end{theorem}

\begin{proof} 
We recall a result of Herzog and Hibi.

\begin{theorem}[Theorem 7.3 \cite{HeHi02}]
The independence polytope $P(M)$ of a loopless matroid $M$ is Gorenstein if and only if there exists a constant $\alpha\in\Z_{\geq 2}$, such that
$$\rank(A)=\frac{1}{\alpha}(\lvert A\rvert+1),$$
for every indecomposable flat $A$, i.e.~a flat that is connected.
\end{theorem}

Suppose a set of edges $A$ forms an indecomposable flat in a graphic matroid $M(G)$. Let $S$ be the set of all endpoints of edges from $A$. As the restriction to the flat is a connected matroid, the pair $(S,A)$ is a $2$-connected multigraph. In particular, it is connected. Since $A$ is a flat, it becomes clear that any edge in $G$ between vertices of $S$ is in $A$. The converse is clear, that is if $S$ induces a $2$-connected subgraph of $G$, then the set $A$ of all edges between vertices $S$ is an indecomposable flat. Identifying an indecomposable flat $A$ with a set of vertices $S$ we obtain: $\alpha(|S|-1)=|E(S)|+1$. 

Assume $G$ is Gorenstein. Applying the above equality to any two vertices that are connected by an edge in $G$ we see that there are exactly $\alpha-1$ parallel edges between them. Hence, $G$ is indeed an $(\alpha-1)$-blow up of a graph $H$. Taking $\alpha=\delta$ we see that $H$ satisfies the desired equalities, as induced subgraphs of $H$ are $(\alpha-1)$-blow ups of induced subgraphs of $G$. The opposite implication follows the same way.
\end{proof}

\begin{theorem}\label{caseP}
	Let $H$ be a $2$-connected graph, and let $\delta\geq 2$. The following conditions are equivalent:
	\begin{enumerate}
		\item $H$ satisfies  $(\clubsuit)_{\delta}$,
		\item $H$ is $\delta$-chordal (any cycle without a chord has exactly $\delta+1$ elements) and has no $K_4$ minor,
		\item $H$ can be constructed from the clique $K_2$, by at each step adding a new $(\delta+1)$-cycle to an edge of the preceding graph.
	\end{enumerate} 
\end{theorem}

\begin{proof}
$(1)\Rightarrow(2)$ Consider a chordless cycle $C$. It is an indecomposable flat in $H$. We have $(\delta-1)|C|+1=\delta(|C|-1)$. This implies $|C|=\delta+1$. 

We now prove by induction on the number of vertices of $H$ that it is $K_4$-minor free. When $H$ is a single edge it is trivial. Let $H'$ be inclusion maximal proper $2$-connected induced subgraph. Clearly, $H'$ satisfies $(\clubsuit)_{\delta}$, hence by induction $H'$ is $K_4$-minor free. There exists an edge $e_1\in E(G)\setminus E(G')$, that has one vertex in $H'$. As $H$ is $2$-connected we may pick a shortest path $p=(e_1,\dots,e_q)$ in $H$, that finishes also with a vertex in $H'$. As $H'$ union $p$ is $2$-connected induced subgraph of $H$, larger than $H'$, it cannot be proper. Thus, $H$ is $H'$ union $p$. Applying $(\clubsuit)_{\delta}$ to $H'$ and $H$ we obtain $q=\delta$. As $H$ is $\delta$-chordal we see that it is obtained from $H'$ by attaching path $p$ to an edge. In particular, since $H'$ is $K_4$-minor free, so is $H$.	
	
$(2)\Rightarrow (3)$ We prove by induction on the number of edges that if $H$ is $2$-connected, $\delta$-chordal, and has no $K_4$ minor, then either $H=K_2$ or $H$ contains an attached $(\delta+1)$-cycle. The graph $H$ has no $K_4$ minor, so it has treewidth at most $2$. Thus, $H$ has a vertex of degree at most $2$. If $H\neq K_2$, then since $H$ is $2$-connected it has a vertex of degree $2$. Let $e,f$ be edges incident to such a vertex. Consider a $2$-connected component of $H\setminus\{e,f\}$. It satisfies inductive assumptions, so it contains an attached $(\delta+1)$-cycle to an edge. This is also an attached $(\delta+1)$-cycle to an edge in $H$, as otherwise minor $K_4$ would appear.
	
$(3)\Rightarrow (1)$ When $H$ is constructed from $H'$ by attaching a $(\delta+1)$-cycle $C$ to edge $e$, there are two types of indecomposable flats, that is $2$-connected induced subgraphs: $2$-connected induced subgraphs of $H'$, and $2$-connected induced subgraphs of $H'$ containing $e$ together with the whole cycle $C$. In both cases equalities $(\clubsuit)_{\delta}$ hold.
\end{proof}

\section{Graphical translation of $\delta$-Gorenstein property for polytope $B(M(G))$}\label{3}

\begin{theorem}\label{TranslationB}
	Fix a positive integer $\delta$. Let $G=(V,E)$ be a $2$-connected graph. The polytope $B(M(G))$ is $\delta$-Gorenstein if and only if 
	$G$ possess the weight function $w:E\rightarrow\{1,\delta-1\}$ defined by
	
	$w(e) = \left\{ \begin{array}{rcl} 
	1 & \mbox{if} & G\setminus e\text{ is }2\text{-connected,} \\
	\delta-1 & \mbox{if} & G/e\text{ is }2\text{-connected,}
	\end{array}\right.$ \newline
	and the following equalities $(\spadesuit)_{\delta}$ are satisfied:
	\begin{enumerate}
		 \item $w(E)=\delta(\lvert V\rvert-1)$, and \emph{(recall the notion $w(E)=\sum_{e\in E} w(e)$)}
		 \item $w(E(S))+1=\delta(\lvert S\rvert-1)$ for every \emph{good flat} $S\subset V$, i.e.~a set such that both: restriction of $G$ to $S$ and contraction of $E(S)$ in $G$ are $2$-connected.
	\end{enumerate}
\end{theorem}

\begin{proof}
	 Recall the general matroid base polytope facets description.
	
	\begin{lemma}[\cite{Ki10,FS05}]\label{lem:scianyBM}
		Let $M$ be a connected matroid on the ground set $E$ with the rank function $r$. Then, the base polytope $B(M)$ is full dimensional in an affine sublattice of $\Z^E$ given by $\sum_{e\in E} x_e=r$, and all supporting hyperplanes (ergo facets) of $B(M)$ are of one of the following two types:
		\begin{enumerate}
			\item $x_e\geq 0$, if $M\setminus\{e\}$ is connected,
			\item $\sum_{e\in F} x_e\leq \frac{r(F)}{r(E)}\sum_{e\in E} x_e$, where $F\subset E$ is a \emph{good flat} -- a flat such that both: restriction of $M$ to $F$ and contraction of $F$ in $M$ are connected.
		\end{enumerate}
	\end{lemma}
	
	We extract from Lemma \ref{lem:scianyBM} the facet presentation of the polytope $B(M(G))$.
	
	\begin{corollary}\label{cor:scianyBG}
		Let $G=(V,E)$ be a $2$-connected graph. The polytope $B(M(G))$ has two types of supporting hyperplanes (ergo facets):
		\begin{enumerate}
			\item $x_e\geq 0$, if $G\setminus e$ is $2$-connected,
			\item $\sum_{e\in G|_{S}} x_e\leq \frac{|S|-1}{|V|-1}\sum_{e\in E} x_e$, where $S\subset V$ is a \emph{good flat} -- a flat such that both: restriction of $G$ to $S$ and contraction of $E(S)$ in $G$ are $2$-connected.
		\end{enumerate}
	\end{corollary}

\begin{proof}	
	Inequalities $(1)$ correspond directly to those from Lemma \ref{lem:scianyBM}. 
	
	For $(2)$ we need to show a correspondence between good flats in a graphic matroid and good flats in a graph. Suppose that a set of edges $E'$ forms a good flat in a graphic matroid. Let $S$ be the set of all endpoints of edges from $E'$. As the restriction to the flat is a connected matroid, the pair $(S,E')$ is a $2$-connected graph. In particular, it is connected. Since $E'$ is a flat, it becomes clear that any edge of $G$ between vertices of $S$ is in $E'$. Thus, restriction of $G$ to $S$ and contraction of $E(S)$ in $G$ are $2$-connected -- $S$ is a good flat in $G$. The converse is clear.
\end{proof}

Now, we proceed to the proof of Theorem \ref{TranslationB}. Weight function $w$ is supposed to correspond to lattice point $v$ from the Definition \ref{def:charGor}. Namely, if weight function $w$ exists, then (by Corollary \ref{cor:scianyBG}) properties of $w$ together with equalities $(\spadesuit)_{\delta}$ guarantee that $v\in\delta P$. Further, we claim that both inequalities in Corollary \ref{cor:scianyBG} provide \emph{reduced} equations of the facets. The equalities $(\spadesuit)_{\delta}$ imply that for every supporting hyperplane of the cone over $P$, its reduced equation $h$ satisfies $h(v)=1$. Therefore, the polytope $B(M(G))$ is $\delta$-Gorenstein.

Conversely, suppose the polytope $B(M(G))$ is $\delta$-Gorenstein. If $G\setminus e$ is $2$-connected, then by Corollary \ref{cor:scianyBG} $x_e\geq 0$ is a supporting hyperplane, and hence $v_e=1$. Otherwise, 
there is no cycle in which $e$ is a chord, so contraction of $e$ does not influence cycles (cycle $C$ becomes a cycle $C\setminus e$), so $G/e$ is $2$-connected. Then, $e$ is a good flat, and (by Corollary \ref{cor:scianyBG}) $x_e\leq \frac{|S|-1}{|V|-1}\sum_{e\in E} x_e$ is a supporting hyperplane, and hence $1+v_e=\frac{1}{|V|-1}\sum_{e\in E} v_e=\frac{1}{|V|-1}(|V|-1)\delta$, so $v_e=\delta-1$. Hence, weight function $w:=v$ exists, and (by Corollary \ref{cor:scianyBG}) properties of $v$ imply equalities $(\spadesuit)_{\delta}$.
\end{proof}

\begin{theorem}\label{TranslationB2}
	Fix a positive integer $\delta$. Let $G=(V,E)$ be a $2$-connected graph. The polytope $B(M(G))$ is $\delta$-Gorenstein if and only if the equality $(\heartsuit)_{\delta}$ is satisfied:
$$w(E(S))+k(S)=\delta(\lvert S\rvert-1)\text{ for every }2\text{-connected set }S\subset V,$$
	where $w:E\rightarrow\{1,\delta-1\}$ is the weight function defined in Theorem \ref{TranslationB}, and $k(S)$ is the number of $2$-connected components in $G/G(S)$ (notice that $k(V)=0$).
\end{theorem}

\begin{proof}
	The implication $\Leftarrow$ is straightforward as $S=V$ gives the first equality of $(\spadesuit)_{\delta}$ and for any good flat $S$ we obtain the second equality of $(\spadesuit)_{\delta}$ with $k(S)=1$.
	
	For the implication $\Rightarrow$ let $C_1,\dots,C_k$ be $2$-connected components in $G/G(S)$ (where $k=k(S)$). Sets $V\setminus (C_i\setminus S)$ are good flats in $G$, thus
	$$w(E(V\setminus (C_i\setminus S)))+1=\delta(\lvert V\setminus (C_i\setminus S)\rvert-1).$$
	Summing up
	$$(k-1)\cdot w(E)+w(E(S))+k=\delta((k-1)\lvert V\rvert+\lvert S\rvert-(k-1)-1),$$
	$$w(E(S))+k=\delta(\lvert S\rvert-1).$$
\end{proof}

\section{Characterization of $\delta$-Gorenstein polytopes $B(M(G))$ for $\delta>2$}\label{4}

Let $\delta>2$ be fixed. The following two propositions show how to construct new Gorenstein graphs from those already known.

\begin{proposition}\label{construction1}
	Suppose $G_1,\dots,G_{\delta-1}$ are $2$-connected graphs satisfying equalities $(\spadesuit)_{\delta}$ from Theorem \ref{TranslationB}. Let $e_1,\dots,e_{\delta-1}$ be edges of the corresponding graphs with weights equal to $\delta-1$. Then, the \emph{glueing of $G_1,\dots,G_{\delta-1}$ along $e_1,\dots,e_{\delta-1}$}, that is a graph $G$ which is a disjoint union of $G_1,\dots,G_{\delta-1}$ with edges $e_1,\dots,e_{\delta-1}$ unified to a distinguished edge $e$, satisfies equalities $(\spadesuit)_{\delta}$ (and the weight of $e$ is equal to $1$). 
\end{proposition}

\begin{proof}
	The first equality $(\spadesuit)_{\delta}$ is easy to check.
	
	For the second, let $S$ be a good flat in $G$. 
	
	If $S$ is disjoint from the endpoints of $e$, then since $G_{|S}$ is connected, $S$ must be contained in some $G_j$. Further, $G_{|S}=G_{j|S}$ is $2$-connected. If the contraction of $S$ in $G_j$ would lead to a separating vertex, it would also be the case in $G$. Hence, $S$ is a good flat in $G_j$ and satisfies the second equality $(\spadesuit)_{\delta}$. 
	
	The same argument works if $S$ contains just one endpoint $v_1$ of $e$. Indeed, if $S$ was not contained in some $G_j$, then $v_1$ would be a cut vertex of $G_{|S}$.
	
	In the remaining case $e$ is an edge of $G_{|S}$. In particular, contraction of $S$ also contracts $e$. If there were two parts $G_{j_1}$ and $G_{j_2}$ not fully contained in $S$, then contraction of $S$ would be a separating vertex. Thus, we may assume $S$ contains all vertices of $G_1,\dots, G_{\delta-2}$. Contraction of $S$ contracts also these parts, and hence $S\cap V(G_{\delta-1})$ must be a good flat $F$ of $G_{\delta-1}$. We obtain:
	
	$$w(E(S))+1=$$
	$$=\sum_{i=1}^{\delta-2}(w_i(E(G_i))-w_i(e))+(w_{\delta-1}(F)-w_{\delta-1}(e))+w(e)+1=$$
	$$=\sum_{i=1}^{\delta-2}w_i(E(G_i))+w_{\delta-1}(F)-(\delta-1)^2+1+1=$$
	$$=\delta\sum_{i=1}^{\delta-2}(|V(G_i)|-1)+\delta(|F|-1)-(\delta-1)^2+1=$$
	$$=\delta(|S|-1)+\delta+2\delta(\delta-2)-\delta(\delta-1)-(\delta-1)^2+1=$$
	$$=\delta(|S|-1)$$
\end{proof}

\begin{proposition}\label{construction2}
	Suppose $G$ is a $2$-connected graph satisfying equalities $(\spadesuit)_{\delta}$ from Theorem \ref{TranslationB}. Let $e$ be an edge with weight equal to $1$. Then, the \emph{$(\delta-1)$-subdivision of $e$}, that is a graph $G'$ equal to $G$ with $e$ replaced by a path $e_1,\dots,e_{\delta-1}$, satisfies equalities $(\spadesuit)_{\delta}$ (and the weight of each $e_i$ is equal to $\delta-1$). 
\end{proposition}

\begin{proof}
	It is straightforward to check the first equality in $(\spadesuit)_{\delta}$.
	
	For the second, let $S$ be a good flat in $G'$. 
	
	If $S$ does not contain the endpoints of $e$ then, as $G'_{|S}$ is connected, it cannot contain any of the vertices on the added path. In such a case, $S$ may be identified with a good flat of $G$ and the second equality in $(\spadesuit)_{\delta}$ follows.
	
	The same argument applies when $S$ contains precisely one endpoint of $e$. Indeed, then it cannot contain any of the other vertices on the added path, as $G'_{|S}$ is $2$-connected. 
	
	Suppose now that both endpoints of $e$ belong to $S$. By $2$-connectivity of $G'_{|S}$ we have only the following two cases.
\begin{itemize}	
\item	$S$ contains all vertices on the added path. By forgetting the added vertices we have an induced good flat $S'$ in $G$. We have:
	$$w(E(S))+1=w(E(S'))+(\delta-1)^2=\delta(|S'|-1)-1+(\delta-1)^2=\delta(|S|-1).$$	
\item	$S$ contains no inner vertices on the added path. Since $G/S$ is $2$-connected, $S$ must contain all other vertices. In particular $|S|=|V(G)|$. We have:
	$$w(E(S))+1=w(E(G))=\delta(|S|-1).$$
	\end{itemize} 
\end{proof}

\begin{theorem}\label{thm:d>2}
	Let $G$ be a $2$-connected graph and let $\delta>2$. The following conditions are equivalent:
	\begin{enumerate}
		\item $G$ satisfies equalities $(\spadesuit)_{\delta}$ from Theorem \ref{TranslationB},
		\item $G$ can be obtained using constructions described in \ref{construction1}, \ref{construction2} from a $\delta$-cycle.
	\end{enumerate} 
\end{theorem}

\begin{proof}
	It is easy to check that $\delta$-cycle satisfies equalities $(\spadesuit)_{\delta}$. Implication $(2)\Rightarrow (1)$ follows from Propositions \ref{construction1} and \ref{construction2}.
	
	The remaining part of this section is devoted to the proof of the implication $(1)\Rightarrow (2)$. The proof of the theorem and all the following claims is by induction on the size of $V(G)$.
	
	First suppose that $G$ contains an edge with weight $1$. In this case we conclude by the following claim, applying construction inverse to the one from Proposition \ref{construction1}.
	
	\begin{claim}\label{lem:decompedge1}
		Suppose a $2$-connected graph $G$ satisfying equalities $(\spadesuit)_{\delta}$ has an edge $e=uv$ of weight $1$. Then, $G$ is the glueing of graphs $G_1,\dots,G_{\delta-1}$ satisfying equalities $(\spadesuit)_{\delta}$ along edges $e_1,\dots,e_{\delta-1}$, where $G_i$ is a subgraph of $G$ induced on the union of $\{u,v\}$ and a connected component of $G\setminus\{u,v\}$.
	\end{claim}
	
	\begin{proof}
		Let $G_1,\dots,G_k$ be the subgraphs of $G$ inducing the $2$-connected components after the contraction of $e$. Note that subgraphs $G_i$ are $2$-connected, and the number of $2$-connected components after contraction of $G_i$ equals $k(G_i)=k-1$.
		
		By equalities $(\heartsuit)_{\delta}$ from Theorem \ref{TranslationB2} we know that:
		$$w(E(G_i))+k-1=\delta(|V(G_i)|-1).$$
		Summing up we obtain:
		$$\sum_{i=1}^k w(E(G_i))+k(k-1)=\delta(\sum_{i=1}^k|V(G_i)|-k).$$
		Knowing that:
		$$\sum_{i=1}^k w(E(G_i))=w(E(G))+k-1,$$
		$$\sum_{i=1}^k|V(G_i)|)=|V(G)|+2(k-1),$$ 
		$$w(E(G))=\delta(|V(G)|-1)$$
		we get that $k=\delta-1$.
		
		It remains to show that each $G_i$ satisfies equalities $(\spadesuit)_{\delta}$ (note that now $w_i(e)=\delta-1$). The first one follows from $$w(E(G_i))+\delta-2=w_i(E(G_i))=\delta(|V(G_i)|-1).$$ 
		
		For the second equality suppose $S$ is a good flat in $G_i$. If $S$ contains at most one endpoint of $e$, then $S$ is also a good flat in $G$ and the second equality of $(\spadesuit)_{\delta}$ follows. 
		
		Suppose contrary, that is $u,v\in S$. We extend $S$ to a good flat $S'$ in $G$ by adding all vertices not in $G_i$. We have:
		$$w_i(E(S))+1=w(E(S'))+1-(\sum_{j:j\neq i} w_j(E(G_j)))+\delta(\delta-2).$$
		Applying equalities $(\spadesuit)_{\delta}$ this gives:
		$$w_i(E(S))+1=\delta(|S'|-1)-(\sum_{j:j\neq i} \delta(|V(G_j)|-1))+\delta(\delta-2),$$
		$$w_i(E(S))+1=\delta(|S|-1).$$
	\end{proof}
	
	Hence, from now on we may assume that all edges of $G$ have weight $\delta-1$. By \emph{$s$-ear} we mean a path in $G$ of length $s$ whose inner vertices have degree $2$. The following claim characterizes $s$-ears in $G$. 
	
	\begin{claim}\label{lem:pathsind-1}
		Let $G$ be a $2$-connected graph satisfying equalities $(\spadesuit)_{\delta}$ in which the weight of every edge equals $\delta-1$. Let $P=(v_0,\dots,v_s)$ be an $s$-ear in $G$. Then, the number of $2$-connected components of $G\setminus\{v_1,\dots,v_{s-1}\}$ equals $\delta-s$.
	\end{claim}
	
	\begin{proof}
		Let $C_1,\dots, C_k$ be the $2$-connected components of $G\setminus\{v_1,\dots,v_{s-1}\}$. As adding the path $P$ makes the graph $2$-connected, there exists an ordering of components $C_i$ for which:
		\begin{itemize}
			\item if $C_i$ and $C_j$ share a vertex, then $i=j\pm 1$, 
			\item $C_i$ has a common vertex with $C_{i+1}$ and this vertex is unique,
			\item $v_0\in C_1$ and $v_s\in C_k$.
		\end{itemize}
		
		As each $C_i$ is a good flat we have 
		$$(\delta-1)|E(C_i)|+1=\delta(|V(C_i)|-1).$$ 
		Summing up we get:
		$$(\delta-1)\sum_{i=1}^k |E(C_i)|+k=\delta(\sum_{i=1}^k |V(C_i)|-k).$$
		Substituting $|E|-s=\sum_{i=1}^k |E(C_i)|$ and $|V|+k-1-(s-1)=\sum_{i=1}^k |V(C_i)|$ we obtain:
		$$(\delta-1)(|E|-s)+k=\delta(|V|-1-(s-1)).$$
		By the first equality in $(\spadesuit)_{\delta}$ we know that $(\delta-1)|E|=\delta(|V|-1)$, which gives us:
		$$-s(\delta-1)+k=\delta(-s+1).$$
		This is precisely the assertion of the claim.
	\end{proof}
	
	\begin{claim}\label{cor:del-1path}
		Let $G$ be a $2$-connected graph satisfying equalities $(\spadesuit)_{\delta}$ in which the weight of every edge equals $\delta-1$. Then, $(\delta-1)$-ears are complements of inclusion maximal $2$-connected induced proper subgraphs.
	\end{claim}
	
	\begin{proof}
		Clearly, by Claim \ref{lem:pathsind-1} every such path must be a complement of a $2$-connected induced subgraph. Of course, it is not possible to extend such a subgraph to a larger proper $2$-connected subgraph.
		
		For the opposite inclusion assume $B$ is an inclusion maximal $2$-connected induced proper subgraph. Consider a shortest path $P=(v_0,v_1,\dots, v_s)$ in $G$, such that $v_0,v_s\in V(B)$ and $v_i\not\in V(B)$ for $i=1,\dots,s-1$. Such a path exists as $B$ is proper (we may find $v_1$) and $G$ is $2$-connected. 
		
		By maximality of $B$, we must have $V(G)=V(B)\cup\{v_1,\dots,v_{s-1}\}$. $P$ is a degree $2$ path, since if $v_i$ was not of degree $2$ for some $i=1,\dots,s-1$ we would be able to construct a shorter path. 
		
		Since $B$ is $2$-connected, by Claim \ref{lem:pathsind-1}, length of $P$ equals $s=\delta-1$.	
	\end{proof}
	
	Claim \ref{cor:del-1path} allowed us to find many $(\delta-1)$-ears in $G$. Claim \ref{lem:ifnot2con} below tells us that $G$ is constructed from smaller graphs applying operation Proposition \ref{construction2}, provided there is a $(\delta-1)$-ear whose contraction is not $2$-connected. Finally, Claim \ref{lem:goodpathexists} will guarantee existence of such $(\delta-1)$-ears. 
	
	\begin{claim}\label{lem:ifnot2con}
		Suppose a $2$-connected graph $G$ satisfying equalities $(\spadesuit)_{\delta}$ with all edges of weight $\delta-1$ contains an $(\delta-1)$-ear $P$ whose contraction is not $2$-connected. Then, the graph obtained from $G$ by replacing $P$ by a single edge also satisfies $(\spadesuit)_{\delta}$.
	\end{claim}
	
	\begin{proof}
		Firstly, notice that $\delta$-cycle does not contain an $(\delta-1)$-ear whose contraction is not $2$-connected.
		
		Secondly, if $G$ is not a $\delta$-cycle, then endpoints of a $(\delta-1)$-ear are not joined by an edge. Indeed, if they were joined by an edge $e$, the contraction of $e$ would be not $2$-connected, implying $w(e)=1$, which contradicts the assumption.
		
		Let $G'$ be the graph $G$ with $P$ replaced by a single edge $e$. By the above remarks $G'$ is again a simple graph. Consider a good flat $S'$ in $G'$. By assumption $e$ itself is not a good flat, and $w_{G'}(e)=1$. As $G$ and $G'$ have the same topology, it follows that the only nontrivial case to consider is when both endpoints of $e$ are in $S'$.  
		
		Let $S$ be the subset of vertices of $G$ that coincides with $S'$ on $G'$, but also contains all inner vertices of the path $P$. As $G'/S'=G/S$ and the second is $2$-connected, it follows that $S$ is a good flat. Thus:
		$$w_G(E(S))+1=\delta(|S|-1),$$
		$$w_{G'}(E(S'))+(\delta-1)^2-1+1=\delta(|S'|+\delta-2-1),$$
		$$w_{G'}(E(S'))+1=\delta(|S'|-1).$$
	\end{proof}
	
	\begin{claim}\label{lem:goodpathexists}
		Let $G$ be a $2$-connected graph satisfying equalities $(\spadesuit)_{\delta}$ in which the weight of every edge equals $\delta-1$. Suppose $G$ is not a $\delta$-cycle. Then, contraction of every $(\delta-1)$-ear is not $2$-connected. 
	\end{claim}
	
	\begin{proof}
		We proceed by induction on the size of $V(G)$. 
		
		First, we notice that in order to prove the assertion for every $(\delta-1)$-ear, it is enough to show that contraction of some $(\delta-1)$-ear $P$ is not $2$-connected. Indeed, if this is the case, then by Claim \ref{lem:ifnot2con} we can replace $P$ by a single edge (which has weight $1$) receiving a graph $G'$ satisfying equalities $(\spadesuit)_{\delta}$. By Claim \ref{lem:decompedge1} graph $G'$ is the glueing of graphs $G_1,\dots,G_{\delta-1}$ satisfying equalities $(\spadesuit)_{\delta}$ along copies of the edge $e$. Since $(\delta-1)$-ears in $G$ are disjoint (as by Claim \ref{lem:pathsind-1} they are maximal degree $2$ paths), every $(\delta-1)$-ear in $G$ different from $P$, is a $(\delta-1)$-ear in some $G_i$. By induction, either its contraction in $G_i$ is not $2$-connected, and therefore its contraction in $G$ is not $2$-connected, or $G_i$ is a $\delta$-cycle and then contraction of that $(\delta-1)$-ear in $G$ is not $2$-connected as contraction of $P$ is not $2$-connected.
		
		Consider the case when there exists a maximal ear $P$ in $G$ of length $l$ distinct from $\delta-1$. By Claim \ref{lem:pathsind-1} the complement of $P$ in $G$ has $\delta-l$ $2$-connected components $C_1,\dots,C_{\delta-l}$. In particular, the number of $2$-connected components is at least $2$. Let $G'$ be the graph $G$ in which we replace components $C_2,\dots,C_{\delta-l}$ by single edges. Clearly, $G'$ is smaller than $G$. It is straightforward to check that $G'$ is a $2$-connected graph satisfying equalities $(\spadesuit)_{\delta}$ in which the weight of every edge equals $\delta-1$. Moreover, $G'$ is not a $\delta$-cycle. By Claim \ref{cor:del-1path} graph $G'$ has many $(\delta-1)$-ears. Let $P'$ be a $(\delta-1)$-ear contained in $C_1$. By inductive assumption of the claim, the contraction of $P'$ is not $2$-connected in $G'$. Thus, it is not $2$-connected in $G$. By the initial remark, $G$ satisfies the assertion of the claim.  
		
		The remaining case is when all maximal ears in $G$ are of length $\delta-1$. In particular, every edge is a part of a single $(\delta-1)$-ear. If contraction of some of $(\delta-1)$-ear is not $2$-connected we are done by the initial remark. Otherwise, graph $G$ satisfies the following conditions:
		\begin{itemize}
			\item $G$ is $2$-connected,
			\item deletion of every $(\delta-1)$-ear is $2$-connected (by Claim \ref{lem:pathsind-1}),
			\item contraction of every $(\delta-1)$-ear is $2$-connected (by assumption of the case),
			\item $(\delta-1)|E|=\delta(|V|-1)$,
			\item $(\delta-1)|E(S)|+1=\delta(|S|-1)$ for every good flat $S\subset V$.
		\end{itemize}
		Let $G'$ be a graph obtained from $G$ by replacing every $(\delta-1)$-ear by a single edge. Notice that every good flat $S'$ in $G'$ distinct from an edge is achieved from a good flat $S$ in $G$ by replacing its $(\delta-1)$-ears to single edges. Graph $G'$ satisfies assumptions of Claim \ref{lem:dziwny}, and therefore by Claim \ref{lem:dziwny} this case is empty. 
	\end{proof}
	
	\begin{claim}\label{lem:dziwny}
		There is no graph $G(V,E)$ satisfying the following conditions:
		\begin{itemize}
			\item $G$ is $2$-connected,
			\item deletion of every edge is $2$-connected,
			\item contraction of every edge is $2$-connected (thus, every edge is a good flat),
			\item $|E|=\delta(|V|-1)$,
			\item $|E(S)|+1=\delta(|S|-1)$ for every good flat $S\subset V$ distinct from an edge.
		\end{itemize}
	\end{claim}
	
	\begin{proof}
		Suppose such graphs exists, and let $G$ be one of them with minimum size of $V(G)$.
		
		Notice that not every vertex in $G$ is of degree $\delta+1$. Indeed, in such a case the average degree would be:
		$$\frac{2|E|}{|V|}=\frac{2\delta(|V|-1)}{|V|}=\delta+1.$$
		This gives $|V|(\delta-1)=2\delta$. As $\delta-1$ and $\delta$ are coprime, we would have $\delta=3$ and $|V|=3$, which is not possible.
		
		Let $v$ be a vertex in $G$ of degree distinct from $\delta+1$. Let $T$ be an inclusion minimal set of edges incident to $v$ such that the graph $(V,E\setminus T)$ is not $2$-connected. Denote its $2$-connected components by $C_1,\dots,C_k$, for $k\geq 2$.
		
		Let $H$ be the block graph of this decomposition -- a graph on all $2$-connected components and all cut vertices, with edges between a cut vertex and a component to which it belongs. Clearly, $H$ is acyclic. For $e\in T$, the graph $(V,(E\setminus T)\cup e)$ is $2$-connected, hence $(V,E\setminus T)$ is connected, so $H$ is a tree. Moreover, since for every $e\in T$ the graph $(V,(E\setminus T)\cup e)$ is $2$-connected, the above structure has to be the following (renumbering indices of $C_i$ if necessary): 
		\begin{itemize}
			\item $H$ is a path: $C_i$ has one common vertex with $C_{i+1}$ for $i=1,\dots,k-1$,
			\item $v\in C_1$,
			\item every $e\in T$ joins $v$ with $C_k$.
		\end{itemize} 
		
		Notice that no component $C_i$ is a single edge. Indeed, if $k>2$ and $C_i$ was an edge, its deletion would be not $2$-connected, contradicting the assumption. If $k=2$ and $C_2$ was an edge, $G$ would have a vertex of degree $2$, contradicting the assumption. Finally, if $k=2$ and $C_1$ was an edge, then from equalities $|E(C_2)|+1=\delta(|V(C_2)|-1)$ (as $C_2$ is a good flat) and $|E|=\delta(|V|-1)$ we would get that $v$ is of degree $\delta+1$, which is not the case.
		
		Thus, every component $C_i$ is a good flat distinct from an edge, and we have the equation
		$$|E(C_i)|+1=\delta(|V(C_i)|-1).$$ 
		Summing up we obtain
		$$|E(G)|-|T|+k=\delta(|V(G)|+(k-1)-k).$$ 
		As $|E(G)|=\delta(|V(G)|-1)$ we have $|T|=k$. Further, $C_k\cup\{v\}$ is also a good flat distinct from an edge, hence 
		$$|E(C_k)|+|T|+1=\delta(|V(C_k)|+1-1).$$ 
		Thus $|T|=k=\delta$. 
		
		Let $G'$ be the graph $G(C_\delta\cup\{v\})$ together with an edge $e'$ between vertices $v$ and $w=V(C_{\delta-1})\cap V(C_\delta)$. In other words, $G'$ is equal to $G$ with components $C_1,\dots,C_{\delta-1}$ replaced by an edge $e'$. Notice that $G'$ is again a simple graph, i.e. in $G$ there was no edge between $v$ and $w$. Indeed, if there was such an edge $e$, then since $C_\delta$ is larger than $e$, the contraction of $e$ would be not $2$-connected.
		
		Now, $G'$ satisfies all conditions of the claim:
		\begin{itemize}
			\item $G'$ is $2$-connected, as $G(C_\delta\cup\{v\})$ is,
			\item deletion of every edge is $2$-connected, as deletion of every edge in $G$ is $2$-connected, and $G(C_\delta\cup\{v\})$  is $2$-connected (for deletion of $e'$),
			\item contraction of every edge is $2$-connected, as in $G$ is, and $G(C_\delta)$ is $2$-connected (for contraction of $e'$),
			\item $|E(G')|=\delta(|V(G')|-1)$, as \newline 
			$|E(G')|=|E(C_\delta)|+1+\delta=\delta(|V(C_\delta)|-1)+\delta=\delta(|V(G')|-1)$,
			\item $|E(S')|+1=\delta(|S'|-1)$ for every good flat $S'$ in $G'$ distinct from an edge, as every good flat in $G'$ that does not contain $e'$ identifies with a good flat in $G$ and the equality holds, or if $S'$ contains $e'$ then $S'$ together with all vertices from $C_1,\dots,C_{\delta-1}$ forms a good flat $S$ in $G$, and we have:
			$$|E(S')|+1=|E(S)|+1-\sum_{i=1}^{\delta-1}|E(C_i)|-(\delta-1)+\delta=$$
			$$=\delta(|V(S)|-1-\sum_{i=1}^{\delta-1}(|V(C_i)|-1)+1)=\delta(|V(S')|-1).$$
		\end{itemize}
		
		Of course, $G'$ is smaller than $G$. Contradiction.
	\end{proof}
	
\end{proof}

\section{Characterization of $\delta$-Gorenstein polytopes $B(M(G))$ for $\delta=2$}\label{5}

We begin with the following simplification of equalities $(\spadesuit)_2$ from Theorem \ref{TranslationB}.

\begin{corollary}\label{cor:char d=2}
	Let $G=(V,E)$ be a $2$-connected graph. The polytope $B(M(G))$ is $2$-Gorenstein if and only if the following equalities $(\spadesuit)_2$ are satisfied:
	\begin{enumerate}
		\item $|E|=2(|V|-1)$, and
		\item $|E(S)|=2|S|-3$ for every good flat $S\subset V$.
	\end{enumerate}
\end{corollary}

\begin{proof}
Suppose $B(M(G))$ is $2$-Gorenstein. Notice that since $G$ is $2$-connected, for every edge $e$ one of the graphs $G\setminus e,G/e$ is $2$-connected. Thus, in Theorem \ref{TranslationB} we must have $w(e)=1$ for every edge $e$. Further, substituting $\delta=2$ we obtain equalities from the assertion.
	
For the opposite implication, define $w(e):=1$. Now $G$ satisfies all conditions from Theorem \ref{TranslationB} and $B(M(G))$ is $2$-Gorenstein.
\end{proof}

\begin{proposition}\label{construction3}
Suppose $G_1,G_2$ are $2$-connected graphs satisfying equalities $(\spadesuit)_2$ from Theorem \ref{TranslationB}. The \emph{collision of $G_1,G_2$ on the corresponding edges $e_1,e_2$}, that is the graph obtained by glueing of $G_1,G_2$ along edges $e_1,e_2$ and removing these edges, satisfies equalities $(\spadesuit)_2$. 
\end{proposition}

\begin{proof}
	The first equality $(\spadesuit)_2$ (using Corollary \ref{cor:char d=2}) is easy to check. 
	
	For the second, let $S$ be a good flat in $G$ -- the collision of $G_1$ and $G_2$ on the corresponding edges $e_1,e_2$. Denote by $v_1,v_2$ vertices in $G$ that come from both $G_1$ and $G_2$.
	
	First suppose that at most one vertex, $v_1$ or $v_2$, belongs to $S$. Then $S$ must be contained either in $G_1$ or $G_2$, as otherwise it would not be $2$-connected. In this case, without loss of generality, $S\subset G_1$ and we claim that $S$ a good flat in $G_1$. Indeed, it is clearly $2$-connected and contraction of $S$ in $G_1$ cannot lead to a separating vertex, as it would also in contraction of $S$ in $G$. Hence $S$ is a good flat in $G_1$ and the second equality $(\spadesuit)_2$ holds.
	
	Now suppose that $v_1,v_2\in S$. Then $S$ would be separating, unless $S$ contains $G_1$ or $G_2$. Thus without loss of generality we may assume $G_1\subset S$. Further, $S\cap V(G_2)$ is a good flat in $G_2$. Thus $|S|=|V(G_1)|-2+|S\cap V(G_2)|$, and $|E_G(S)|=$ 
	$$|E(G_1)|+|E_{G_{2}}(S\cap V(G_2))|-2=2(|V(G_1)|-1)+2|S\cap V(G_2)|-3-2=2|S|-3.$$
\end{proof}

\begin{theorem}\label{2-characterization}
Let $G$ be a $2$-connected graph. The following conditions are equivalent:
	\begin{enumerate}
		\item $G$ satisfies $(\spadesuit)_2$ from Theorem \ref{TranslationB},
		\item $G$ can be obtained using construction described in \ref{construction3} from the clique $K_4$.
	\end{enumerate} 
\end{theorem}

\begin{proof}

Since $K_4$ satisfies $(\spadesuit)_2$ implication $(2)\Rightarrow (1)$ is a consequence of Proposition \ref{construction3}. 

The remaining part of this section is devoted to the proof of the implication $(1)\Rightarrow (2)$.

\begin{claim}\label{3-connected}
Suppose $G$ is a $3$-connected graph satisfying $(\spadesuit)_2$. Then, $G=K_4$.
\end{claim}

\begin{proof}
For a vertex $v\in V$ the graph $G\setminus\{v\}$ is 2-connected, so $V\setminus\{v\}$ is a good flat for every $v\in V$. Therefore, the number of edges in $G\setminus\{v\}$ equals $2|V|-5$. On the other hand, the number of edges in $G\setminus \{v\}$ is $|E|-\deg(v)=2|V|-2-\deg(v)$, which implies $\deg(v)=3$ for every vertex $v\in V$. This yields to $$2|V|-2=|E|=\frac{3}{2}|V|,$$ which implies $|V|=4$, and $G=K_4$.
\end{proof}

In the remaining cases $G$ is not $3$-connected. Therefore, we may assume that there exist two separating vertices $v_1,v_2 \in V$. Denote the \emph{connected} components of $G\setminus\{v_1,v_2\}$ by $\SC_1,\dots,\SC_n$, where $n\ge 2$. Every component $\SC_i$ must be joined with both $v_1$ and $v_2$, since none of them is a separating vertex of the graph $G$. 

\begin{lemma}\label{connectedcomponents}
 Let $G$ be a $2$-connected graph satisfying $(\spadesuit)_2$. Suppose that $\{v_1,v_2\}$ is separating set and that $\SC_1,\dots,\SC_n$ are connected components of $G\setminus\{v_1,v_2\}$. Then the following holds:
\begin{enumerate}
\item $n=2$ (after removing $v_1,v_2$ there are exactly two connected components),
\item there is no edge between $v_1$ and $v_2$,
\item $G\setminus\SC_i$ is $2$-connected for every $i$.
\end{enumerate}
\end{lemma}

\begin{proof}
We prove $(1)$ by contradiction. Suppose that $n\ge 3$.

Suppose that $G\setminus \SC_i$ has a separating vertex $v$. However, $v$ is not separating vertex in $G$, since $G$ is $2$-connected. Therefore, $v$ separates vertices $v_1$ and $v_2$. This is impossible since they are connected by a component $\SC_j$ such that $j\neq i$, $v\not\in \SC_j$.

This means that $G\setminus \SC_i$ is $2$-connected. Consequently $G\setminus \SC_i$ is a good flat, since $G$ is $2$-connected and $\SC_i$ is (by definition) connected (implying $G/(G\setminus\SC_i)$ is $2$-connected).

We sum the number of edges in $G\setminus \SC_i$. Every edge in $G$ is counted exactly $n-1$ times except of the edge between $v_1$ and $v_2$ which is counted $n$ times, if it exists. Therefore:

$$(n-1)|E|\le\sum_{i=1}^n|E(G\setminus \SC_i)|=\sum_{i=1}^n (2|V(G\setminus \SC_i)|-3)=$$
$$=\sum_{i=1}^n (2|V|-2|\SC_i|-3)=2(n-1)|V|+4-3n.$$
 We get an inequality
 $$2(n-1)|V|-2(n-1)\le 2(n-1)|V|+4-3n,$$
 $$n\le 2,$$
 
which is a contradiction. Thus, $n=2$ and we proved $(1)$.

For $(2)$ observe that if sets $G\setminus \SC_i$ are $2$-connected (which is the statement of part $(3)$) we can do the same calculation and conclude that $v_1$ and $v_2$ are not joint by an edge. Thus, it remains to prove $(3)$.

For $(3)$ notice first, that if there is an edge between $v_1$ and $v_2$ (that is if the statement of $(2)$ is false), then $(3)$ is clear since $G$ is $2$-connected. Otherwise, denote by $C_1^1,\dots,C_k^1$ $2$-connected components of $G\setminus \SC_1$ and by $C_1^2,\dots,C_l^2$ $2$-connected components of $G\setminus \SC_2$. If $G\setminus \SC_1$ is not $2$-connected, then vertices $v_1$ and $v_2$ cannot lie in the same component $C^1_i$ because that would mean that also $G$ is not $2$-connected. Thus, every $C_j^i$ is a good flat (as it is $2$-connected and its contraction does not destroy $2$-connectivity of $G$). The fact that the block graph of $2$-connected components is a tree (in our case in fact a path) implies 
$$\sum_{j=1}^k|V(C_j^1)|=|V(G\setminus\SC_1)|+k-1.$$
We may again compute the edges assuming $(2)$ is true: 
$$2|V|-2=|E|=\sum_{j=1}^k|E(C_j^1)|+\sum_{j=1}^l|E(C_j^2)|=$$
$$=\sum_{j=1}^k(2|V(C_j^1)|-3)+\sum_{j=1}^l(2|V(C_j^2)|-3)=$$
$$=2\sum_{j=1}^k|V(C_j^1)|-3k+2\sum_{j=1}^l|V(C_j^2)|-3l=$$
$$=2|V(G\setminus\SC_1)|+2k-2-3k+2|V(G\setminus\SC_2)|+2l-2-3l=2|V|-k-l.$$
We conclude $k+l=2$ which gives $k=l=1$ and that $G\setminus \SC_i$ are 2-connected.
\end{proof}

Now we are ready to prove the implication $(1)\Rightarrow(2)$ from Theorem \ref{2-characterization}. The proof is by induction on the number of vertices of a graph. 

Let $G$ be a $2$-connected graph satisfying $(\spadesuit)_2$. Suppose that the statement is true for all graphs with smaller number of vertices than $G$.

If $G$ is $3$-connected, then $G=K_4$ by Lemma \ref{3-connected}. 

Otherwise, $G$ is not $3$-connected and has two separating vertices $v_1$ and $v_2$ which by Lemma \ref{connectedcomponents} are not joined by an edge. Using the notation from Lemma \ref{connectedcomponents} define $G_i$ as the graph induced by $V\setminus \SC_i$ with an additional edge between $v_1$ and $v_2$. It follows that the graph $G$ is the collision of graphs $G_1$ and $G_2$ on the edge $\{v_1,v_2\}$. 

Now, it is sufficient to show that graphs $G_1,G_2$ are $2$-connected and satisfy conditions $(\spadesuit)_2$, since then by induction hypothesis both can be obtained by construction from Proposition \ref{construction3} and so can $G$. Cases for $G_1$ and for $G_2$ are analogous. Lemma \ref{connectedcomponents} gives that $G_1$ is $2$-connected. Since $V\setminus\SC_1$ is a good flat 
$$|E(G_1)|=1+2|V(G_1)|-3=2|V(G_1)|-2,$$ 
so $G_1$ has the right number of edges. It remains to check that every good flat has the right number of edges. Let $S\subset V\setminus \SC_1$ be a good flat in $G_1$. There are 2 cases:

\begin{itemize}
\item $\{v_1,v_2\}\not \subset S$. Then $S$ is also a good flat in the graph $G$ because the component $\SC_1$ is joined with both $v_1$ and $v_2$. Consequently, the number of edges in the induced graph ${G_1}|_{S}$ is $2|S|-3$.
\item $v_1,v_2\in S$. Then the graph induced by $S\cup \SC_1$ is still $2$-connected because both $G|_{S}$ and $G\setminus \SC_2$ are $2$-connected. It follows that $S\cup \SC_1$ is a good flat in $G$ and the number of edges in the induced subgraph of $G$ is $2(|S|+|\SC_1|)-3$. Clearly, $|E_G(S\cup\SC_1)|=|E(G\setminus\SC_2)|+|E_{G_1}(S)|-1$ which yields to $|E_{G_1}(S)|=2|S|-3$.
\end{itemize}

\end{proof}


\end{document}